\documentclass[12pt]{amsart}
\usepackage{amsmath,latexsym,amsfonts,amssymb,amsthm}
\usepackage{geometry}
\usepackage{hyperref}
\usepackage{mathrsfs}
\usepackage{graphicx,color}
\usepackage[alphabetic]{amsrefs}
\usepackage[all,cmtip]{xy}

\newtheorem{prop}{Proposition}[section]
\newtheorem{thm}[prop]{Theorem}
\newtheorem{cor}[prop]{Corollary}

\newtheorem{lem}[prop]{Lemma}

\theoremstyle{definition}

\newtheorem{defn}[prop]{Definition}
\newtheorem{expl}[prop]{Example}
\newtheorem{rem}[prop]{\it Remark}

\newtheorem*{claim*}{Claim}

\newcommand{\bP}{\mathbb{P}}
\newcommand{\bC}{\mathbb{C}}
\newcommand{\bR}{\mathbb{R}}
\newcommand{\bA}{\mathbb{A}}
\newcommand{\bQ}{\mathbb{Q}}
\newcommand{\bZ}{\mathbb{Z}}
\newcommand{\bN}{\mathbb{N}}

\newcommand{\bG}{\mathbb{G}}

\newcommand{\tX}{\widetilde{X}}

\newcommand{\tdelta}{\widetilde{\delta}}

\newcommand{\cX}{\mathcal{X}}

\newcommand{\cO}{\mathcal{O}}
\newcommand{\cL}{\mathcal{L}}
\newcommand{\cI}{\mathcal{I}}
\newcommand{\cM}{\mathcal{M}}
\newcommand{\cF}{\mathcal{F}}
\newcommand{\cJ}{\mathcal{J}}

\newcommand{\cW}{\mathcal{W}}

\newcommand{\fa}{\mathfrak{a}}
\newcommand{\fb}{\mathfrak{b}}

\newcommand{\kbar}{\bar{k}}
\newcommand{\Xkb}{X_{\bar{k}}}
\newcommand{\Deltakb}{\Delta_{\bar{k}}}

\newcommand{\Vkb}{V_{\bar{k}}}
\newcommand{\Ykb}{Y_{\bar{k}}}
\newcommand{\Zkb}{Z_{\bar{k}}}

\newcommand{\ra}{\mathrm{a}}

\newcommand{\Spec}{\mathbf{Spec}}
\newcommand{\Supp}{\mathrm{Supp}}

\newcommand{\mult}{\mathrm{mult}}
\newcommand{\lct}{\mathrm{lct}}
\newcommand{\Ex}{\mathrm{Ex}}
\newcommand{\Pic}{\mathrm{Pic}}
\newcommand{\Aut}{\mathrm{Aut}}

\newcommand{\ord}{\mathrm{ord}}
\newcommand{\Val}{\mathrm{Val}}
\newcommand{\Nklt}{\mathrm{Nklt}}
\newcommand{\Gr}{\mathrm{Gr}}
\newcommand{\Fut}{\mathrm{Fut}}
\newcommand{\Gal}{\mathrm{Gal}}
\newcommand{\Cosupp}{\mathrm{Cosupp}}

\newcommand{\Dtc}{\Delta_{\mathrm{tc}}}
\newcommand{\ocX}{\overline{\mathcal{X}}}
\newcommand{\ocL}{\overline{\mathcal{L}}}
\newcommand{\oDtc}{\overline{\Delta}_{\mathrm{tc}}}

\numberwithin{equation}{section}

\title{Optimal destabilizing centers and equivariant K-stability}

\author{Ziquan Zhuang}
\address{Department of Mathematics, MIT, Cambridge, MA, 02139.}
\email{ziquan@mit.edu}

\date{}

\begin{document}

\maketitle

\begin{abstract}
    We give an algebraic proof of the equivalence of equivariant K-semistability (resp. equivariant K-polystability) with geometric K-semistability (resp. geometric K-polystability). Along the way we also prove the existence and uniqueness of minimal optimal destabilizing centers on K-unstable log Fano pairs.
\end{abstract}

\section{Introduction}

K-stability (see \cites{Tian-K-stability-defn,Don-K-stability-defn}) is an algebraic condition that detects the existence of K\"ahler-Einstein metric on complex Fano varieties. By a well-known result of Matsushima \cite{Mat-aut-of-KE-Fano}, on a K\"ahler-Einstein Fano manifold, the canonical metric is invariant under some maximal compact subgroup of the automorphism group. This motivates the folklore conjecture that a complex Fano variety (or more generally, a log Fano pair) with a group action is K-semistable (resp. K-polystable) if and only if it is equivariantly K-semistable (resp. equivariantly K-polystable). This conjecture has been confirmed on smooth Fano manifolds with a reductive group action \cite{DS-equivariant-KE} and on log Fano pairs with a finite group action \cite{LZ-equivariant}. Despite the algebraic nature of the statement, the proofs in these two cases rely on deep analytic results and an algebraic proof of the conjecture is only known when the group action is given by a torus \cites{LX-valuation-stability,LWX-Kps-degeneration}. On the other hand, if the log Fano pair is defined over a smaller field $k$ (not necessarily algebraically closed), it comes with an induced Galois action. A variant of the above conjecture then predicts that it is enough to test the K-semistability (resp. K-polystability) of the log Fano pair using test configurations that are defined over $k$. To our knowledge, there has been very little progress in this direction.

In this paper, we give a complete answer to these conjectures using purely algebraic argument. Our main result is as follows.

\begin{thm}[=Corollary \ref{cor:equivariant}] \label{main:equivariant K=K}
Let $k$ be a field of characteristic zero and let $\kbar$ be its algebraic closure. Let $G$ be an algebraic group and let $(X,\Delta)$ be a log Fano pair with an action of $G$ over $k$. Let $(\Xkb,\Deltakb):=(X,\Delta)\times_k \Spec(\kbar)$.
\begin{enumerate}
    \item If $(X,\Delta)$ is $G$-equivariantly K-semistable, then $(\Xkb,\Deltakb)$ is K-semistable.
    \item If $G$ is reductive and $(X,\Delta)$ is $G$-equivariantly K-polystable, then $(\Xkb,\Deltakb)$ is K-polystable.
\end{enumerate}
\end{thm}

Note that in the K-semistable part we allow the group $G$ to be non-reductive; this can happen for automorphism groups of K-semistable Fano varieties, see \cite{CS-KE-V22}*{Example 1.4}. On the other hand, there exist Fano varieties that are equivariantly K-polystable with respect to their automorphism group but not K-polystable (see Example \ref{ex:non-reductive}), thus the reductivity assumption in the K-polystable part is necessary. Since every K-polystable log Fano pair has a reductive automorphism group \cite{ABHLX}, this seems to be a natural assumption to add to the statement.

In fact, we prove a more precise statement than Theorem \ref{main:equivariant K=K}. Recall that the K-semistability of a log Fano pair is characterized by its stability threshold (or $\delta$-invariant), see \cites{FO-delta,BJ-delta} or Section \ref{sec:prelim-beta & delta}. A priori, to define the stability threshold of a log Fano pair $(X,\Delta)$, one needs to take into account all divisors on various birational models of $X$. We show that in the K-unstable case, it is enough to consider divisors that not only have the same field of definition but also are invariant under the automorphism group. This is the key to our proof of Theorem \ref{main:equivariant K=K}.

\begin{thm}[=Theorem \ref{thm:G-delta=delta}] \label{main:G-delta=delta}
Let $(X,\Delta)$ be a log Fano pair defined over a field $k$ of characteristic zero and let $G=\Aut(X,\Delta)$. Assume that $(\Xkb,\Deltakb)$ is not K-semistable. Then we have 
\[
\delta(\Xkb,\Deltakb)=\inf_E \frac{A_{X,\Delta}(E)}{S(E)}
\]
where the infimum runs over all $G$-invariant geometrically irreducible divisors $E$ over $X$ that are lc places of complements.
\end{thm}

Here $A_{X,\Delta}(E)$ is the log discrepancy of $E$ with respect to the pair $(X,\Delta)$ and $S(E)$ is the expected vanishing order of $-(K_X+\Delta)$ along the divisor $E$, see Definition \ref{defn:invariants}. As a corollary, we obtain an algebraic proof of a generalization of Tian's criterion \cites{Tian-criterion,OS-alpha} (see Section \ref{sec:prelim-beta & delta} for the definition of the $G$-alpha invariant $\alpha_G(X,\Delta)$ of a log Fano pair).

\begin{cor}[=Corollary \ref{cor:Tian's criterion}] \label{main:Tian's criterion}
Let $G$ be an algebraic group and let $(X,\Delta)$ be a log Fano pair of dimension $n$ with a $G$-action over a field $k$ of characteristic zero. 
\begin{enumerate}
    \item If $\alpha_G(X,\Delta)\ge \frac{n}{n+1}$, then $(\Xkb,\Deltakb)$ is K-semistable.
    \item If $G$ is reductive and $\alpha_G(X,\Delta) > \frac{n}{n+1}$, then $(\Xkb,\Deltakb)$ is K-polystable.
\end{enumerate}
\end{cor}

As another application, we recover the K-stability of Fermat hypersurfaces, as well as smooth complete intersections of two quadrics \cites{Tian-Fermat,AGP-two-quadric}. We remark that the previous proof of this fact involves analytic argument but our proof here is completely algebraic. Combined with \cites{Xu-quasimonomial,BLX-openness}, this also gives an algebraic proof that a general Fano hypersurface is K-semistable (in fact K-stable if the degree is at least $3$).

\begin{cor}[=Corollary \ref{cor:fermat and two quadric}] \label{main:fermat and two quadric}
The following Fano manifolds are K-stable:
\begin{enumerate}
    \item Fermat hypersurfaces $(x_0^d+\cdots+x_{n+1}^d=0)\subseteq \bP^{n+1}$ $(3\le d\le n+1)$.
    \item Complete intersection of two quadrics $Q_1\cap Q_2\subseteq \bP^{n+2}$.
\end{enumerate}
\end{cor}

Along the proof of Theorem \ref{main:G-delta=delta}, we also discover the following result concerning optimal destabilizing centers (i.e. centers of valuations that compute the stability thresholds, see Definition \ref{defn:optimal destab center}) of K-unstable log Fano pairs, which may be of independent interest.

\begin{thm}[=Theorem \ref{thm:minimal delta<1 center}] \label{main:minimal delta<1 center}
Let $(X,\Delta)$ be a log Fano pair. Assume that $(X,\Delta)$ is not K-semistable. Then there exists a $($necessarily unique$)$ subvariety $Z\subseteq X$ such that
\begin{enumerate}
    \item $Z$ is an optimal destabilizing center of $(X,\Delta)$ and
    \item $Z$ is contained in any other optimal destabilizing center of $(X,\Delta)$.
\end{enumerate}
\end{thm}

Let us briefly explain some key ideas behind the proofs of Theorems \ref{main:G-delta=delta} and \ref{main:minimal delta<1 center}. By definition, an optimal destabilizing center of a log Fano pair $(X,\Delta)$ is (roughly speaking) an lc center of $(X,\Delta+\delta D)$ for some basis type $\bQ$-divisor $D$ of the pair $(X,\Delta)$, where $\delta=\delta(X,\Delta)$ is the stability threshold. Our first observation is that in fact every pair of optimal destabilizing centers can be realized as lc centers of a common pair $(X,\Delta+\delta D)$. This allows us to show that optimal destabilizing centers exhibit properties that are similar to lc centers of a fixed lc pair. In particular, any two optimal destabilizing centers intersect (Lemma \ref{lem:delta<1 center intersect}) due to Koll\'ar-Shokurov's connectedness theorem and the intersection is a union of optimal destabilizing centers (Lemma \ref{lem:intersection of two delta center}). As a result, the minimal optimal destabilizing center is unique and contained in any other optimal destabilizing centers. If the log Fano pair is defined over $k$ and admits an action of an algebraic group $G$, then this unique center is also defined over $k$ and invariant under the $G$-action. While the valuations that compute the stability threshold may not be $G$-invariant a priori, we can at least identify a $G$-invariant one $v$ that computes the ``equivariant version'' of the stability threshold at the minimal optimal destabilizing center. When $v$ is divisorial and realized by some prime divisor $E$ over $X$, we may relate the stability threshold (and its geometric or equivariant version) on $X$ to a similar invariant on $E$ through the recent work \cite{AZ-K-adjunction}. This suggests us to use induction on the dimension. However, as the divisor $E$ is not necessarily log Fano and the valuation $v$ may not even be divisorial, we need to carefully set up the inductive framework and work in a setting that's slightly more general than those of \cite{AZ-K-adjunction}, see Theorem \ref{thm:G-delta=delta for boundary}.

This paper is organized as follows. In Section \ref{sec:prelim} we recall some definitions and preliminary results. Theorem \ref{main:minimal delta<1 center} is proved in Section \ref{sec:destab centers} where we study the behaviour of optimal destabilizing centers. In Section \ref{sec:equiv K=K}, we prove Theorem \ref{main:G-delta=delta} and deduce Theorem \ref{main:equivariant K=K}, Corollaries \ref{main:Tian's criterion} and \ref{main:fermat and two quadric} as a consequence.

\subsection*{Postscript note}

Since this paper appeared on the arXiv, there has been many progress on the algebraic study of K-stability that leads to different proofs of some results in this article. Firstly, it is shown in \cite{XZ-uniqueness} that the minimizer of the normalized volume function is unique (up to rescaling) for any klt singularity. Combined with the cone construction argument (see e.g. \cite{Li-equivariant-minimize}*{Theorem 3.1} and \cite{Zhu-equivariant-K}*{Theorem B}), this gives a different proof of our Theorem \ref{main:equivariant K=K}(1). More recently, \cite{LXZ-HRFG} proves the optimal destabilization conjecture. Combined with \cite{BHLLX-theta}, it implies that to any K-unstable Fano variety one can associate a unique destabilizing test configuration that minimizes the functional defined in \cite{BHLLX-theta}. Such a test configuration is necessarily induced by an invariant divisor. This gives another proof of Theorem \ref{main:G-delta=delta} and further shows that the infimum in Theorem \ref{main:G-delta=delta} is in fact a minimum.

\subsection*{Acknowledgement}

The author would like to thank Yuchen Liu and Chenyang Xu for many helpful discussions, and Ivan Cheltsov and Xiaowei Wang for suggestions and comments. He would also like to thank the anonymous referee for careful reading of the manuscript and for the many helpful suggestions and comments.

\section{Preliminaries} \label{sec:prelim}

\subsection{Notation and conventions}

Throughout the paper let $k$ be a field of characteristic zero and let $\kbar$ be its algebraic closure. Unless otherwise specified, all varieties, morphisms and linear series are assumed to be defined over $k$. Given an object $X$ (e.g. a variety/divisor/linear series) over $k$ and a field extension $k\subseteq K$, we denote by $X_K$ the corresponding base change. A pair $(X,\Delta)$ consists of a normal geometrically irreducible variety $X$ and an effective $\bQ$-divisor $\Delta$ such that $K_X+\Delta$ is $\bQ$-Cartier. The notions of klt and lc singularities are defined as in \cite{Kol-mmp}*{Definition 2.8}. A pair $(X,\Delta)$ is log Fano if $X$ is projective, $(X,\Delta)$ is klt and $-(K_X+\Delta)$ is ample. The non-klt center $\Nklt(X,\Delta)$ of a pair $(X,\Delta)$ is the union of points $x\in X$ such that $(X,\Delta)$ is not klt at $x$. If $\pi:Y\to X$ is a projective birational morphism and $E$ is a prime divisor on $Y$, then we say $E$ is a divisor over $X$. We denote by $C_X(E)$ the center of $E$ on $X$. Let $(X,\Delta)$ be a klt pair, $Z\subseteq X$ a subvariety and $D$ an effective divisor on $X$, we denote by $\lct_Z(X,\Delta;D)$ the largest number $\lambda\ge 0$ such that $(X,\Delta+\lambda D)$ is lc at the generic point of $Z$ (by convention, we set $\lct_Z(X,\Delta;D)=+\infty$ if $Z$ is not contained in the support of $D$). A $\bQ$-ideal on $X$ is a formal linear combination $\fa=\prod_{i=1}^m\fa_i^{\lambda_i}$ where $\fa_i\subseteq \cO_X$ are ideal sheaves on $X$ and $\lambda_i\in\bQ_+$. Its co-support $\Cosupp(\fa)$ is defined to be the union of $\Supp(\cO_X/\fa_i)$. We can similarly define the log canonical threshold $\lct_Z(X,\Delta;\fa)$ of a $\bQ$-ideal $\fa$ with respect to the pair $(X,\Delta)$. We denote by $\Val_X$ the set of $\bR$-valued valuations on the function field $k(X)^*$ that is trivial on $k^*$ and has a center on $X$. The log discrepancy $A_{X,\Delta}(v)$ of a valuation $v\in\Val_X$ with respect to a pair $(X,\Delta)$ is defined as in \cite{JM-seq-of-ideal}*{(5.2)} and we denote by $\Val^*_X$ the set of $v\in\Val_X$ with $A_{X,\Delta}(v)<\infty$ for some $\Delta$ such that $K_X+\Delta$ is $\bQ$-Cartier (it does not depend on the choice of $\Delta$). When $X$ has a group $G$ action, a valuation $v$ is said to $G$-invariant if $v(g\cdot s)=v(s)$ for any $g\in G$ and any $s\in k(X)^*$; a divisor $E$ over $X$ is $G$-invariant if the associated valuation $\ord_E$ is $G$-invariant. Given a $\bQ$-divisor $D$ on $X$, we set
\[
H^0(X,D):=\{0\neq s\in k(X)\,|\,{\rm div}(s)+D\ge 0\}\cup \{0\}
\]
whose members can be viewed as effective $\bQ$-divisors that are $\bZ$-linearly equivalent to $D$. If $D$ is $\bQ$-Cartier, then $v(s):=v({\rm div}(s)+D)$ is well-defined for any $0\neq s\in H^0(X,D)$ and any valuation $v\in\Val_X$; we denote by $\cF_v$ the induced filtration, i.e.
\[
\cF^\lambda_v H^0(X,D) := \{s\in H^0(X,D)\,|\,v(s)\ge \lambda\}.
\]

\subsection{Test configurations and K-stability}

In this section, we recall the definition of K-stability and equivariant K-stability.

\begin{defn}[\cites{Tian-K-stability-defn,Don-K-stability-defn,OS-alpha,LX-special-tc}]
Let $(X,\Delta)$ be an $n$-dimensional log Fano pair with an action of an algebraic group $G$. Let $L$ be an ample line bundle on $X$ such that $L\sim -r(K_X+\Delta)$ for some $r\in\bN^*$.
\begin{enumerate}
\item A (normal) \emph{$G$-equivariant test configuration} $(\cX,\Dtc;\cL)/\bA^1$ of $(X,\Delta;L)$ consists of the following data:
\begin{itemize}
 \item a normal variety $\cX$, an effective $\bQ$-divisor $\Dtc$ on $\cX$, together with a flat projective morphism $\pi\colon \cX\to \bA^1$;
 \item a $\pi$-ample line bundle $\cL$ on $\cX$;
 \item a $G\times \bG_m$-action on $(\cX,\Dtc;\cL)$ such that $\pi$ is $G\times\bG_m$-equivariant with respect to the trivial action of $G$ on $\bA^1$ and the standard action of $\bG_m$ on $\bA^1$ via multiplication;
 \item $(\cX\setminus\cX_0,\Dtc|_{\cX\setminus\cX_0};\cL|_{\cX\setminus\cX_0})$
 is $G\times \bG_m$-equivariantly isomorphic to $(X,\Delta;L)\times(\bA^1\setminus\{0\})$.
\end{itemize}

\item
A $G$-equivariant test configuration is called a \emph{product} test configuration if
\[
(\cX,\Dtc;\cL)\cong(X\times\bA^1,\Delta\times\bA^1;{\rm pr}_1^* L).
\]
A product test configuration is called a \emph{trivial} test configuration if the above isomorphism is $G\times\bG_m$-equivariant with respect to the given $G$-action on $X$, trivial $\bG_m$-action on $X$, trivial $G$-action on $\bA^1$, and the standard $\bG_m$-action on $\bA^1$ via multiplication.

\item
A $G$-equivariant test configuration $(\cX,\Dtc;\cL)$ is said to be \emph{special} if $(\cX,\cX_0+\Dtc)$ is plt and $\cL\sim_\bQ -r(K_{\cX}+\Dtc)$. In this case, we say that $(\cX_0,(\Dtc)_0)$ (which is necessarily log Fano) is a \emph{$G$-equivariant special degeneration} of $(X,\Delta)$.

\item (cf. \cites{Wang12, Odaka13})
Assume $\pi:(\cX,\Dtc;\cL)\to \bA^1$ is a $G$-equivariant test configuration of 
$(X,\Delta;L)$. Let $\bar{\pi}: (\ocX,\oDtc;\ocL)\to\bP^1$ be the natural $G\times\bG_m$-equivariant compactification of $\pi$. The \emph{generalized Futaki invariant} of $(\cX,\Dtc;\cL)$ is defined by the intersection formula
\[
\Fut(\cX,\Dtc;\cL):=\frac{1}{(-K_X-\Delta)^n}\left(\frac{n}{n+1}\cdot\frac{(\ocL^{n+1})}{r^{n+1}}+\frac{(\ocL^n\cdot (K_{\ocX/\bP^1}+\oDtc))}{r^n}\right).
\]
\item The log Fano pair $(X,\Delta)$ is said to be
\begin{itemize}
  \item \emph{$G$-equivariantly K-semistable} if $\Fut(\cX,\Dtc;\cL)\geq 0$ for any $G$-equivariant  test configuration $(\cX,\Dtc;\cL)/\bA^1$ and any $r\in\bN^*$ such that $L$ is Cartier. 
  \item \emph{$G$-equivariantly K-polystable} if it is $G$-equivariantly K-semistable and for any $G$-equivariant test configuration $(\cX,\Dtc;\cL)/\bA^1$ we have $\Fut(\cX,\Dtc;\cL)=0$ if and only if it is a product test configuration.
  \item \emph{geometrically K-semistable} (resp. \emph{geometrically K-polystable}) if $(\Xkb,\Deltakb)$ is equivariantly K-semistable (resp. K-polystable) over $\kbar$ with respect to the trivial group action.
\end{itemize}
\end{enumerate}
\end{defn}

Taking $G$ to be the trivial group and $k=\bC$, we recover the usual definition of the K-semistability (resp. K-polystablity) of a complex log Fano pair. We say that a log Fano pair $(X,\Delta)$ is K-unstable if it is not K-semistable. The following fact will be frequently used in this paper.

\begin{thm}[\cite{LWX-Kps-degeneration}] \label{thm:test G-Kps by stc}
Let $(X,\Delta)$ be a geometrically K-semistable log Fano pair with an action of a group $G$. Then it is $G$-equivariantly K-polystable if and only if every geometrically K-semistable $G$-equivariant special degeneration of $(X,\Delta)$ is isomorphic to $(X,\Delta)$.
\end{thm}

\begin{proof}
Let $(\cX,\Dtc;\cL)$ be a $G$-equivariant test configuration such that $\Fut(\cX,\Dtc;\cL)=0$. Since $(X,\Delta)$ is geometrically K-semistable, the test configuration is special by \cite{LX-special-tc}*{Theorem 7} and the central fiber $(X_0,\Delta_0)$ is geometrically K-semistable by \cite{LWX-Kps-degeneration}*{Lemma 3.1}. The result then follows as $(\cX,\Dtc;\cL)$ is a product test configuration if and only if $(X_0,\Delta_0)\cong (X,\Delta)$ (an isotrivial family over $\bA^1$ is automatically trivial).
\end{proof}

\subsection{Valuative criterion and stability thresholds} \label{sec:prelim-beta & delta}

We recall the valuative criterion of K-semistability developed by Fujita and Li as well as the definition of stability threshold. 

\begin{defn}[\cites{FO-delta,BJ-delta}] \label{defn:invariants}
Let $(X,\Delta)$ be a log Fano pair and let $L=-(K_X+\Delta)$. Let $m>0$ be an integer such that $H^0(X,mL)\neq 0$.
\begin{enumerate}
    \item An $m$-basis type $\bQ$-divisor of $(X,\Delta)$ is a divisor of the form
    \[
    D=\frac{1}{mN_m}\sum_{i=1}^{N_m} \{s_i=0\}
    \]
    where $N_m=h^0(X,mL)$ and $s_1,\cdots,s_{N_m}$ is a basis of $H^0(X,mL)$. We define $\delta_m(X,\Delta)$ to be the largest number $\lambda\ge 0$ such that $(X,\Delta+\lambda D)$ is lc for every $m$-basis type $\bQ$-divisor $D$ of $(X,\Delta)$.
    \item Let $v\in\Val^*_X$ be a valuation. We define the following invariants:
    \begin{align*}
        T_m(v) & =\frac{\max\{v(D)\,|\,D\in |mL|\}}{m},\\
        S_m(v) & =\max\{v(D)\,|\,D\text{ is an }m\text{-basis type }\bQ\text{-divisor of }(X,\Delta)\}.
    \end{align*}
    We set $T(v)=\lim_{m\to\infty} T_m(v)$, $S(v)=\lim_{m\to\infty} S_m(v)$. If $E$ is a divisor over $X$, we define $S(E):=S(\ord_E)$, $T(E):=T(\ord_E)$, etc.
    \item The stability threshold (or $\delta$-invariant) of $(X,\Delta)$ is defined to be
    \[
    \delta(X,\Delta):=\inf_{v\in\Val^*_X} \frac{A_{X,\Delta}(v)}{S(v)}.
    \]
    By \cite{BJ-delta}*{Theorem A}, we have $\lim_{m\to \infty} \delta_m(X,\Delta)=\delta(X,\Delta)$.
\end{enumerate}
\end{defn}

\begin{defn}[\cite{Tian-criterion} and \cite{CS-lct-3fold}*{Appendix}]
Let $(X,\Delta)$ be a log Fano pair with an action of a group $G$. The $G$-alpha invariant $\alpha_G(X,\Delta)$ of $(X,\Delta)$ is defined to be the largest $\lambda\ge 0$ such that $(X,\Delta+\frac{\lambda}{m}\cM)$ is lc for any $m\in\bN^*$ and any $G$-invariant linear system $\cM\subseteq |-m(K_X+\Delta)|$. By definition, it is clear that
\[
\alpha_G(X,\Delta) \le \frac{A_{X,\Delta}(v)}{T(v)}
\]
for any $G$-invariant valuation $v\in \Val^*_X$.
\end{defn}

\begin{thm}[\cites{Fuj-valuative-criterion,Li-equivariant-minimize,FO-delta,BJ-delta}]
Let $(X,\Delta)$ be a log Fano pair. Assume that $k=\kbar$ is algebraically closed. The following are equivalent:
\begin{enumerate}
    \item $(X,\Delta)$ is K-semistable.
    \item $\beta_{X,\Delta}(E):=A_{X,\Delta}(E)-S(E)\ge 0$ for any divisor $E$ over $X$.
    \item $\delta(X,\Delta)\ge 1$.
\end{enumerate}
\end{thm}

We also need the following equivariant version of the above valuative criterion.

\begin{defn}[\cite{Fuj-valuative-criterion}] \label{defn:dreamy}
Let $(X,\Delta)$ be a log Fano pair and let $E$ be a divisor over $X$. Let $r>0$ be an integer such that $L:=-r(K_X+\Delta)$ is ample. We say that $E$ is dreamy if the graded algebra $\bigoplus_{m,i\in\bN} \cF_E^i H^0(X,\cO_X(mL))$ is finitely generated.
\end{defn}

\begin{rem} \label{rem:dreamy}
If there exists some effective $\bQ$-divisor $D\sim_\bQ -(K_X+\Delta)$ such that $(X,\Delta+D)$ is lc and $A_{X,\Delta+D}(E)=0$ (in this case we say that $E$ is an \emph{lc place of complement}), then $E$ is dreamy. In fact, $(X,\Delta+(1-\epsilon)D)$ is log Fano for some $0<\epsilon\ll 1$ and $0<A_{X,\Delta+(1-\epsilon)D}(E)<1$, thus by \cite{BCHM}*{Corollary 1.4.3} there exists a proper birational morphism $\pi\colon Y\to X$ with unique exceptional divisor $E$. Using the crepant pullback of $(X,\Delta+(1-\epsilon)D)$, it is not hard to see that $Y$ is of Fano type and hence the graded algebra $\bigoplus_{m,i\in\bN} \cF_E^i H^0(X,\cO_X(mL))=\bigoplus_{m,i\in\bN} H^0(Y,\cO_Y(m\pi^*L-iE))$ is finitely generated by \cite{BCHM}*{Corollary 1.3.2}.
\end{rem}

\begin{prop} \label{prop:G-valuative criterion}
Let $(X,\Delta)$ be a log Fano pair with an action of an algebraic group $G$. Assume that $(X,\Delta)$ is $G$-equivariantly K-semistable. Then $A_{X,\Delta}(E)\ge S(E)$ for any $G$-invariant geometrically irreducible dreamy divisor $E$ over $X$.
\end{prop}

\begin{proof}
By \cite{Fuj-valuative-criterion}*{Theorem 5.2 and Section 6}, the divisor $E$ induces a test configuration $(\cX,\Dtc;\cL)$ of $(X,\Delta)$ with geometrically integral central fiber such that 
\[
\Fut(\cX,\Dtc;\cL)=A_{X,\Delta}(E)-S(E).
\]
Since $E$ is $G$-invariant, the test configuration is $G$-equivariant, hence $\Fut(\cX,\Dtc;\cL)\ge 0$ as $(X,\Delta)$ is $G$-equivariantly K-semistable and the result follows.
\end{proof}

Finally we define the notion of optimal destabilizing centers.

\begin{defn} \label{defn:optimal destab center}
Let $(X,\Delta)$ be a log Fano pair. A valuation $v\in \Val^*_X$ is called a $\delta$-minimizing (resp. destabilizing) valuation of $(X,\Delta)$ if
\[
\delta(X,\Delta)=\frac{A_{X,\Delta}(v)}{S(v)} \quad\text{resp.}\quad \frac{A_{X,\Delta}(v)}{S(v)}<1.
\]
A subvariety $Z\subseteq X$ is called a $\delta$-minimizing (resp. destabilizing) center of $(X,\Delta)$ if there exists a $\delta$-minimizing (resp. destabilizing) valuation $v$ of $(X,\Delta)$ with center $Z$. When $\delta(X,\Delta)<1$, a $\delta$-minimizing valuation (resp. center) of $(X,\Delta)$ is also called an optimal destabilizing valuation (resp. center).
\end{defn}

\begin{rem}
By the proof of \cite{BLX-openness}*{Theorem 4.5} (which doesn't require the base field to be algebraically closed), optimal destabilizing valuations (resp. centers) exist on every log Fano pair $(X,\Delta)$ with $\delta(X,\Delta) < 1$. In general, $\delta$-minimizing valuations (resp. centers) are only known to exist over an uncountable base field by the generic limit argument of \cite{BJ-delta}*{Theorem E}. 
\end{rem}

\subsection{Graded sequence of ideals}

A graded sequence of ideals (see \cite{JM-seq-of-ideal} for a general discussion) on a variety is a sequence of ideals $\fa_\bullet=(\fa_m)_{m\in \bN}$ such that $\fa_m\cdot \fa_n\subseteq \fa_{m+n}$ for all $m,n\in\bN$. As a typical example, every valuation $v\in\Val_X$ gives rise to a graded sequence of ideals $\fa_\bullet (v)$ as follows: let $U\subseteq X$ be an affine open set; we set $\fa_m(v)(U):=\{f\in\cO_X(U)\,|\,v(f)\ge m\}$ if $v$ has a center in $U$ and otherwise $\fa_m(v)(U):=\cO_X(U)$. 

Given $v\in\Val_X$ and a graded sequence of ideals $\fa_\bullet$, we can evaluate $\fa_\bullet$ along $v$ by setting
\[
v(\fa_\bullet)=\inf_{m\in\bN^*} \frac{v(\fa_m)}{m}.
\]
Let $(X,\Delta)$ be a sub-pair (i.e. $K_X+\Delta$ is $\bQ$-Cartier but $\Delta$ is not necessarily effective), let $\fa_\bullet$ be a graded sequence of ideals and let $\lambda\in \bR$. The pair $(X,\Delta+\fa^\lambda_\bullet)$ is said to be klt (resp. lc) if $A_{X,\Delta}(v) > \lambda\cdot v(\fa_\bullet)$ (resp. $A_{X,\Delta}(v)\ge \lambda\cdot v(\fa_\bullet)$) for all $v\in\Val^*_X$. Note that if $(X,\Delta+\fa^{\lambda/m}_m)$ is klt (resp. lc) for some $m\in\bN^*$ then the same is true for $(X,\Delta+\fa^\lambda_\bullet)$. A subvariety $Z\subseteq X$ is called a non-klt center of $(X,\Delta+\fa^\lambda_\bullet)$ if there exists a valuation $v\in \Val^*_X$ with center $Z$ such that $A_{X,\Delta}(v)\le \lambda\cdot v(\fa_\bullet)$ (equivalently, $A_{X,\Delta}(v)\le \lambda\cdot \frac{v(\fa_m)}{m}$ for all $m\in\bN^*$).

\begin{lem} \label{lem:non-klt center seq of ideals}
Let $(X,\Delta)$ be a sub-pair, let $\fa_\bullet$ be a graded sequence of ideals on $X$ and let $\lambda\in \bR$. Assume that $\bigcap_{m=1}^\infty \Nklt(X,\Delta+\fa^{\lambda/m}_m)\neq \emptyset$ and let $S$ be an irreducible component of the intersection. Then there exists a valuation $v\in\Val^*_X$ with center $S$ such that 
\[
\lambda\cdot v(\fa_\bullet)\ge A_{X,\Delta}(v).
\]
In other words, $S$ is a non-klt center of $(X,\Delta+\fa^\lambda_\bullet)$.
\end{lem}

\begin{proof}
This essentially follows from \cite{JM-seq-of-ideal}. The hypothesis and statement are unaffected by taking log resolution and crepant pullbacks etc. so we may assume that $X$ is smooth and $\Delta$ has geometric simple normal crossing support. Localizing at the generic point of $S$ and replacing $\fa_\bullet$ by $(\fa_{m\ell})_{m\in\bN}$ for some $\ell\gg 0$, we may also assume that $S$ is a point and $(X,\Delta+\fa^{\lambda/m}_m)$ is klt away from $S$. In particular, $D=-\lfloor \Delta \rfloor\ge 0$. For simplicity, we further assume that $\{\Delta\}=0$ (since the result and argument of \cite{JM-seq-of-ideal} easily extends to the log smooth case). Recall that the multiplier ideal (see e.g. \cite{Laz-positivity-2}) $\cJ(\fa^{\lambda/m}_m)$ consists of those $f\in\cO_X$ such that
\[
\ord_E(f)>\frac{\lambda}{m}\ord_E(\fa_m)-A_X(E)
\]
for all divisors $E$ over $X$. Let $\cJ(\fa_\bullet)=\cup_{m\in\bN^*} \cJ(\fa^{\lambda/m}_m)$ which equals $\cJ(\fa^{\lambda/m}_m)$ for sufficiently large $m$. By assumption, we also have $A_X(E)+\ord_E(D)\le \frac{\lambda}{m}\ord_E(\fa_m)$ for some divisor $E$ as $(X,-D+\fa^{\lambda/m}_m)$ is not klt. Thus $\cO_X(-D)\not\subseteq \cJ(\fa^{\lambda/m}_m)$ for all $m\in\bN^*$ and therefore 
\[
\cO_X(-D)\not\subseteq \cJ(\fa_\bullet).
\]
In the notation of \cite{JM-seq-of-ideal}*{Section 1.4} this means $\lct^{\cO_X(-D)}(\fa_\bullet)\le \lambda$ and by \cite{JM-seq-of-ideal}*{Theorem 7.3}, there exists a valuation $v\in\Val^*_X$ such that 
\[
\lambda\cdot v(\fa_\bullet)\ge A_X(v)+v(D)=A_{X,\Delta}(v).
\]
Since $(X,\Delta+\fa^\lambda_\bullet)$ is klt outside $S$, the valuation $v$ is necessarily centered at $S$. This completes the proof.
\end{proof}

\section{Optimal destabilizing centers} \label{sec:destab centers}

In this section, we prove the existence and uniqueness of minimal optimal destabilizing centers (Definition \ref{defn:optimal destab center}) of a K-unstable log Fano pair.

\begin{thm} \label{thm:minimal delta<1 center}
Let $(X,\Delta)$ be a log Fano pair with $\delta(X,\Delta)<1$. Then there exists a $($necessarily unique$)$ subvariety $Z\subseteq X$ such that
\begin{enumerate}
    \item $Z$ is an optimal destabilizing center of $(X,\Delta)$ and
    \item $Z$ is contained in any other optimal destabilizing center of $(X,\Delta)$.
\end{enumerate}
\end{thm}

Note that the result is false for $\delta$-minimizing centers in general if we allow $(X,\Delta)$ to be K-semistable: for example, every closed point on $\bP^n$ is a $\delta$-minimizing center.

As we will see in the proof, optimal destabilizing centers behave like non-klt centers of a graded sequence of ideal. For this reason we need the following property of such non-klt centers. 

\begin{lem} \label{lem:intersection of two non-klt center}
Let $(X,\Delta)$ be a pair, let $(\fa_i)_{i\in\bN^*}$ be a graded sequence of ideals and let $\lambda\in\bR_{\ge 0}$. Let $Z_1,Z_2\subseteq X$ be non-klt centers of $(X,\Delta+\fa^\lambda_\bullet)$. Assume that $Z_1\cap Z_2\neq \emptyset$ and $(X,\Delta+\fa^\lambda_\bullet)$ is klt outside $Z_1\cup Z_2$. Then $Z_1\cap Z_2$ is a union of non-klt centers of $(X,\Delta+\fa^\lambda_\bullet)$.
\end{lem}

\begin{proof}
We may assume that $Z_1\not\subseteq Z_2$ and $Z_2\not\subseteq Z_1$, otherwise there is nothing to prove. Passing to an open neighbourhood of a generic point of $Z_1\cap Z_2$, we may assume that $X$ is affine and $Z:=Z_1\cap Z_2$ is irreducible. Let $\pi\colon Y\to X$ be a log resolution of $(X,\Delta)$, let $W$ (resp. $W_i$, $i=1,2$) be the union of $\pi$-exceptional divisors whose centers are contained in $Z_1\cap Z_2$ (resp. contained in $Z_i$ but not in $Z_1\cap Z_2$). Passing to a higher resolution, we may assume that $W_1$ is disjoint from $W_2$ (this can be achieved by blowing up strata in $W_1\cap W_2$ until they don't intersect). Let $(Y,\Delta_Y)$ be the crepant pullback of $(X,\Delta)$, i.e. $K_Y+\Delta_Y=\pi^*(K_X+\Delta)$. Note that $\Delta_Y$ is not necessarily effective. For every integer $m>0$, we also let $V_m=\Nklt(Y,\Delta_Y+\pi^*\fa^{\lambda/m}_m)$. Since $\fa_\bullet$ is a graded sequence of ideals, we have $V_{m\ell}\subseteq V_m$ for all $m,\ell\in\bN^*$. In particular $V_1\supseteq V_2\supseteq \cdots \supseteq V_{2^m}\supseteq \cdots$ and thus we have $V_{2^m}=V_{2^{m+1}}=\cdots=V$ for all sufficiently large $m$. By Koll\'ar-Shokurov connectedness theorem \cite{K+-flip}*{Theorem 17.4} (applied to the pair $(Y,\Delta_Y+\frac{\lambda}{m} \pi^*D)$ over $X$, where $D=\{f=0\}$ and $f\in\fa_m$ is a general member), we know that $V_m$ is connected over the generic point of $Z$ and thus $V$ is connected over the generic point of $Z$ as well. By assumption and \cite{JM-seq-of-ideal}*{Theorem A}, $(X,\Delta+\fa^{\lambda/m}_m)$ is klt outside $Z$ for sufficiently large $m$ but fail to be klt along either $Z_i$ for any $m\in\bN^*$. It follows that $V\subseteq W\cup W_1\cup W_2$ and $V\cap W_i\neq \emptyset$ for both $i=1,2$. Since $W_1$ is disjoint from $W_2$, we deduce that $V\cap W\neq \emptyset$ (over the generic point of $Z$) and hence any irreducible component $S$ of $V\cap W$ that dominates $Z$ is a non-klt center of $(Y,\Delta_Y+f^*\fa^\lambda_\bullet)$ by Lemma \ref{lem:non-klt center seq of ideals}. Since $\pi(S)=Z$ by construction and $(Y,\Delta_Y+f^*\fa^\lambda_\bullet)$ is the crepant pullback of $(X,\Delta+\fa^\lambda_\bullet)$, we conclude that $Z$ is a non-klt center of $(X,\Delta+\fa^\lambda_\bullet)$.
\end{proof}

Our proof of Theorem \ref{thm:minimal delta<1 center} now divides into two steps. We first show that the intersection of two optimal destabilizing centers is a union of optimal destabilizing centers (see \cite{Kol-mmp} for similar properties of lc centers).

\begin{lem} \label{lem:intersection of two delta center}
Let $(X,\Delta)$ be a log Fano pair and let $Z_1,Z_2$ be $\delta$-minimizing $($resp. destabili-zing$)$ centers of $(X,\Delta)$. Then $Z_1\cap Z_2$ is a union of $\delta$-minimizing $($resp. destabilizing$)$ centers.
\end{lem}

Before giving the proof, let us recall a definition from \cite{AZ-K-adjunction}.

\begin{defn}[c.f. \cite{AZ-K-adjunction}*{Definition 1.5}]
Let $(X,\Delta)$ be a log Fano pair and let $v\in\Val^*_X$. Let $m\in\bN$ and let $D$ be an $m$-basis type $\bQ$-divisor of $L=-(K_X+\Delta)$, i.e. there exists a basis $s_1,\cdots,s_{N_m}$ of $H^0(X,mL)$, where $N_m=h^0(X,mL)$, such that 
\[
D=\frac{1}{mN_m}\sum_{i=1}^{N_m} \{s_i=0\}.
\]
We say that $D$ is \emph{compatible with} $v$ if $\cF_v^\lambda H^0(X,mL)$ is spanned by some $s_i$ for every $\lambda\in\bR$. It is not hard to see from the definition that $S_m(v)=v(D)$ for any $m$-basis type $\bQ$-divisor $D$ that's compatible with $v$.
\end{defn}

\begin{proof}[Proof of Lemma \ref{lem:intersection of two delta center}]
We only prove the lemma for $\delta$-minimizing centers since the argument is similar for destabilizing centers. By assumption, there exist valuations $v_1,v_2\in\Val^*_X$ with centers $Z_1,Z_2$ such that 
\begin{equation} \label{eq:v_i are delta minimizer}
    \delta(X,\Delta)=\frac{A_{X,\Delta}(v_1)}{S(v_1)}=\frac{A_{X,\Delta}(v_2)}{S(v_2)}.
\end{equation}
Up to rescaling, we may assume that $A_{X,\Delta}(v_1)=A_{X,\Delta}(v_2)=1$. For each integer $m>0$, let
\[
\fa_m:=\fa_m(v_1)\cap \fa_m(v_2)\subseteq \cO_X
\]
where $\fa_m(v_i)$ are the valuation ideals of $v_i$. Then $\fa_\bullet$ is a graded sequence of ideals with co-support $Z_1\cup Z_2$; in particular, $(X,\Delta+\fa^\lambda_m)$ is klt outside $Z_1\cup Z_2$ for any $\lambda>0$. Since $v_i(\fa_\bullet)\ge 1=A_{X,\Delta}(v_i)$ ($i=1,2$), we see that both $Z_i$ are non-klt centers of $(X,\Delta+\fa_\bullet)$, thus by Lemma \ref{lem:intersection of two non-klt center}, for any irreducible component $Z$ of $Z_1\cap Z_2$, there exists a valuation $v\in\Val^*_X$ with center $Z$ such that $v(\fa_\bullet)\ge A_{X,\Delta}(v)$. As before we may assume that $A_{X,\Delta}(v)=1$. Then for any $f\in \cO_{X,Z}$ we have $v(f)\ge \lambda$ if $v_i(f)\ge \lambda$ for both $i=1,2$; in other words, $v(f)\ge \min\{v_1(f),v_2(f)\}$. We claim that $v$ is a $\delta$-minimizing valuation of $(X,\Delta)$. To see this, let $m$ be a sufficiently divisible integer and let $D$ be an $m$-basis type $\bQ$-divisor of $(X,\Delta)$ that's compatible with both $v_i$ (which exists by \cite{AZ-K-adjunction}*{Lemma 3.1}). Then we have $v_i(D)=S_m(v_i)$ ($i=1,2$) and thus $S_m(v)\ge v(D)\ge \min\{S_m(v_1),S_m(v_2)\}$. Letting $m\to \infty$ we obtain $S(v)\ge S(v_1)=S(v_2)=\frac{1}{\delta(X,\Delta)}$ where the equalities follow from \eqref{eq:v_i are delta minimizer}. It follows that $\delta(X,\Delta)\ge \frac{1}{S(v)} = \frac{A_{X,\Delta}(v)}{S(v)}$ but we always have $\delta(X,\Delta)\le \frac{A_{X,\Delta}(v)}{S(v)}$ by definition; therefore $v$ is a $\delta$-minimizing valuation as desired.
\end{proof}

The next step is to show that optimal destabilizing centers intersect with each other using Koll\'ar-Shokurov's connectedness theorem.

\begin{lem} \label{lem:delta<1 center intersect}
Let $Z_1,Z_2$ be optimal destabilizing centers of a K-unstable log Fano pair $(X,\Delta)$. Then $Z_1$ has nonempty intersection with $Z_2$.
\end{lem}

\begin{proof}
Suppose that $Z_1\cap Z_2=\emptyset$, we will derive a contradiction. Let $r$ be a sufficiently large and divisible integer such that $\cI_{Z_1\cup Z_2}\otimes \cO_X(-r(K_X+\Delta))$ is globally generated (where $\cI_{Z_1\cup Z_2}$ denotes the ideal sheaf of $Z_1\cup Z_2$) and fix some $\epsilon$ with $0<r\epsilon<1-\delta(X,\Delta)$. By assumption, there exist valuations $v_1,v_2\in\Val^*_X$ with centers $Z_1,Z_2$ such that 
\[
\delta(X,\Delta)=\frac{A_{X,\Delta}(v_1)}{S(v_1)}=\frac{A_{X,\Delta}(v_2)}{S(v_2)}.
\]
Let $m\gg 0$ be such that $r\epsilon<1-\delta_m(X,\Delta)$ and 
\begin{equation} \label{eq:tiebreak at v_i}
    \epsilon\cdot v_i(\cI_{Z_i})+\delta_m(X,\Delta) S_m(v_i)>A_{X,\Delta}(v_i)
\end{equation}
for both $i=1,2$. Let $D$ be an $m$-basis type $\bQ$-divisor of $(X,\Delta)$ that's compatible with both $v_i$ (which exists by \cite{AZ-K-adjunction}*{Lemma 3.1}) and let $H$ be a general member of $|\cI_{Z_1\cup Z_2}\otimes \cO_X(-r(K_X+\Delta))|$. Then we have $v_i(D)=S_m(v_i)$, hence by \eqref{eq:tiebreak at v_i} we get $A_{X,\Delta+\delta_m D+\epsilon H}(v_i)<0$ (where $\delta_m:=\delta_m(X,\Delta)$). By the definition of $\delta_m$, we also know that $(X,\Delta+\delta_m D)$ is lc, therefore as $H$ is general, we see that $(X,\Delta+(1-\gamma)\delta_m D+\epsilon H)$ is klt away from $Z_1\cup Z_2$ while
\[
A_{X,\Delta+(1-\gamma)\delta_m D+\epsilon H}(v_i)<0
\]
for all $0<\gamma\ll 1$. In other words, $\Nklt(X,\Delta+(1-\gamma)\delta_m D+\epsilon H)=Z_1\cup Z_2$. But as $-(K_X+\Delta+(1-\gamma)\delta_m D+\epsilon H)\sim_\bQ -(1-(1-\gamma)\delta_m-r\epsilon)(K_X+\Delta)$ is ample (by our choice of $m$ and $\epsilon$) and $Z_1$ is disjoint from $Z_2$, this contradicts Koll\'ar-Shokurov's connectedness theorem \cite{K+-flip}*{Theorem 17.4}.
\end{proof}

We are ready to prove Theorem \ref{thm:minimal delta<1 center}.

\begin{proof}[Proof of Theorem \ref{thm:minimal delta<1 center}]
Let $Z\subseteq X$ be a minimal (with respect to inclusion) optimal destabilizing center of $(X,\Delta)$. We claim that it satisfies the statement of the theorem. Let $Z'$ be another optimal destabilizing center of $(X,\Delta)$. By Lemma \ref{lem:delta<1 center intersect}, $Z'$ has nonempty intersection with $Z$; on the other hand, by Lemma \ref{lem:intersection of two delta center}, $Z\cap Z'$ is a union of optimal destabilizing centers of $(X,\Delta)$. Since $Z$ is minimal with respect to inclusion, this implies $Z\subseteq Z'$, which completes the proof.
\end{proof}

\begin{expl}
From the uniqueness it is clear that the minimal optimal destabilizing center constructed in Theorem \ref{thm:minimal delta<1 center} is invariant under the automorphism group of $(X,\Delta)$. This helps us to identify the center in many cases. For example, if $X$ is the blowup of one or two points on $\bP^2$, then there is a unique $\Aut(X)$-invariant $(-1)$-curve on $X$ which is necessarily the minimal optimal destabilizing center.
\end{expl}

\section{Equivariant K-stability} \label{sec:equiv K=K}

In this section, we show that to compute the stability threshold of a geometrically K-unstable log Fano pair, it is enough to use divisors that are defined over the base field and invariant under the automorphism group. It will imply all remaining results mentioned in the introduction. 

\begin{thm} \label{thm:G-delta=delta}
Let $(X,\Delta)$ be a log Fano pair and let $G=\Aut(X,\Delta)$. Assume that $(\Xkb,\Deltakb)$ is not K-semistable. Then we have 
\begin{equation} \label{eq:delta=inf over G-inv div}
    \delta(\Xkb,\Deltakb)=\inf_E \frac{A_{X,\Delta}(E)}{S(E)}
\end{equation}
where the infimum runs over all $G$-invariant geometrically irreducible divisors $E$ over $X$ that are lc places of complements.
\end{thm}

We remark that as in Theorem \ref{thm:minimal delta<1 center}, the statement fails in general if $(\Xkb,\Deltakb)$ is K-semistable; for example, $\bP^n$ does not even have ${\rm PGL}_{n+1}$-invariant divisors over it.

We will prove Theorem \ref{thm:G-delta=delta} by induction on the dimension. However, for the induction to work, we need to slightly generalize the context and we are naturally led to consider boundaries of the following form. Similar objects have already appeared in the work \cite{AZ-K-adjunction} as we refine a linear series by a divisor.

\begin{defn}
Let $(X,\Delta)$ be a pair. A \emph{boundary} on $X$ is a linear combination $V=a_1 V_1+\cdots+a_r V_r$ where $a_i\in\bQ_+$ and $V_i$ are finite dimensional linear series associated to $\bQ$-Cartier divisors (i.e. there exist some effective $\bQ$-Cartier divisors $L_i$ on $X$ such that $V_i\subseteq H^0(X,L_i)$ is a finite dimensional subspace). It's divisor class $c_1(V)\in {\rm NS}(X)_\bQ$ is defined in the obvious way. We say $V$ is $G$-invariant (where $G\subseteq \Aut(X)$ is a subgroup) if the linear series $V_i$ are all $G$-invariant. A \emph{filtration} $\cF$ of $V$ is given by a collection of subspaces $\cF^\lambda V_i \subseteq V_i$ for each $\lambda\in\bR$ and $i=1,\cdots,r$ such that
\begin{enumerate}
    \item $\cF^\lambda V_i \subseteq \cF^{\lambda'} V_i$ whenever $\lambda\ge \lambda'$,
    \item $\cF^\lambda V_i=V_i$ for $\lambda\ll 0$ and $\cF^\lambda V_i=0$ for $\lambda\gg 0$.
\end{enumerate}
As a typical example, every valuation $v\in\Val_X$ induces a filtration $\cF_v$ on $V$. A basis type $\bQ$-divisor of $V$ is a divisor of the form $D=a_1 D_1+\cdots+a_r D_r$ where each $D_i\sim_\bQ L_i$ is a basis type $\bQ$-divisor of $V_i$, i.e.
\[
D_i=\frac{1}{\dim V_i} \sum_{j=1}^{\dim V_i} \{s_j=0\}
\]
where the $s_j$ form a basis of $V_i$. We say that  $D$ is compatible with a filtration $\cF$ on $V$ if every $\cF^\lambda V_i$ is spanned by some $s_j$ in the above expression. 
In particular, we say that $D$ is compatible with a divisor $E$ over $X$ if it is compatible with the filtration $\cF_{\ord_E}$.
\end{defn}

\begin{defn}
Let $U$ be a quasi-projective variety and let $f\colon (X,\Delta)\to U$ be a klt pair that's projective over $U$. Let $V$ be a boundary on $X$, let $E$ be a divisor over $X$ and let $Z\subseteq X$ be a subvariety. We set
\[
S(V;E):=\sup_D \ord_E(D)=\max_D \ord_E(D),
\]
\[
\delta_Z(V)=\delta_Z(X,\Delta;V):=\inf_D \lct_Z(X,\Delta;D)=\min_D \lct_Z(X,\Delta;D)\in \bR\cup \{+\infty\}
\]
where the supremum or infimum runs over all basis type $\bQ$-divisors $D$ of $V$. Clearly $\delta_Z(V)<+\infty$ if and only if there exists some $1\le i\le r$ and some $s\in V_i$ such that $Z\subseteq \{s=0\}$. Note also that as the linear series $V_i$ are finite dimensional, the set of basis type $\bQ$-divisors is bounded, hence the above supremum (resp. infimum) is a  maximum (resp. minimum) by the constructibility of $\ord_E$ (resp. log canonical thresholds) in family. It is easy to see that
\[
\delta_Z(V)=\inf_E \frac{A_{X,\Delta}(E)}{S(V;E)}=\min_E \frac{A_{X,\Delta}(E)}{S(V;E)}
\]
where the infimum runs over all divisors $E$ over $X$ whose center contains $Z$. If $Y\subseteq U$ is a subvariety, we also define the $\delta$-invariant of $V$ (with respect to $(X,\Delta)$) over $Y$ to be
\[
\delta(V/Y)=\delta(X,\Delta;V/Y):=\inf_Z \delta_Z(V)
\]
where the infimum runs over all subvarieties $Z\subseteq X$ whose image in $U$ contains $Y$. In other words, we only take the log canonical thresholds of basis type divisors along the fiber of $f$ over the generic point of $Y$. Let $G$ be an algebraic group. A $G$-action on $(X,\Delta)$ over $U$ is given by a $G$-action on both $(X,\Delta)$ and $U$ making the projection $f$ equivariant.
\end{defn}

We now state a technical result which will imply Theorem \ref{thm:G-delta=delta}. It is specifically designed for inductive purpose and will eventually be applied to the various complete linear series $|-m(K_X+\Delta)|$ of a K-unstable log Fano pair $(X,\Delta)$ to show that all $\delta_m(\Xkb,\Deltakb)$ ($m\gg 0$) are computed by an $\Aut(X,\Delta)$-invariant geometrically irreducible divisor over $X$.

\begin{thm} \label{thm:G-delta=delta for boundary}
Let $G$ be an algebraic group, let $f\colon (X,\Delta)\to U$ be a klt pair that's projective over $U$ with a $G$-action, let $Y\subseteq U$ be a $G$-invariant geometrically irreducible subvariety and let $V$ be a $G$-invariant boundary on $X$. Assume that $\delta(V_{\kbar}/Y_{\kbar})<+\infty$ and that $-(K_X+\Delta+\delta(V_{\kbar}/Y_{\kbar})c_1(V))$ is $f$-ample. Then there exists some $G$-invariant geometrically irreducible divisor $E$ over $X$ whose center dominates $Y$ such that
\[
\delta(V_{\kbar}/Y_{\kbar})=\frac{A_{X,\Delta}(E)}{S(V;E)}.
\]
\end{thm}

We divide the proof of Theorem \ref{thm:G-delta=delta for boundary} into several steps. Using ideas from the previous section, we first show that the it is enough to calculate the geometric $\delta$-invariant at some $G$-invariant geometrically irreducible center.

\begin{lem} \label{lem:G-inv center}
Under the assumptions of Theorem \ref{thm:G-delta=delta for boundary}, there exists some $G$-invariant geometrically irreducible subvariety $Z\subseteq X$ such that $Y\subseteq f(Z)$ and $\delta(V_{\kbar}/Y_{\kbar})=\delta_{Z_{\kbar}}(V_{\kbar})$.
\end{lem}

\begin{proof}
Let $\delta=\delta(V_{\kbar}/Y_{\kbar})$ and let $\eta$ be the generic point of $Y_{\kbar}$. By definition, we have $\lct(\Xkb,\Deltakb;D)=\delta$ in a neighbourhood of $X_\eta$ for some $m$-basis type $\bQ$-divisors $D$ of $\Vkb$. Among such $D$ we choose one (call it $D_0$) such that the minimal lc center (denoted by $Z$) of $(\Xkb,\Deltakb+\delta D_0)$ that intersects $X_\eta$ has the smallest dimension. Clearly $\delta_Z(\Vkb)=\delta$. It remains to show that $Z$ can be defined over $k$ and is $G$-invariant. To this end, let $E_0$ be an lc place of $(\Xkb,\Deltakb+\delta D_0)$ with $C_{\Xkb}(E_0)=Z$, let $g\in \Gal(\kbar/k)\times G(\kbar)$ and let $E_1=g(E_0)$ (viewed as a divisor over $\Xkb$). Since $V$ is $G$-invariant, we get $A_{\Xkb,\Deltakb}(E_0)=A_{\Xkb,\Deltakb}(E_1)$ and $S(\Vkb;E_0)=S(\Vkb;E_1)$. By \cite{AZ-K-adjunction}*{Lemma 3.1}, we can choose a basis type $\bQ$-divisor $D$ of $\Vkb$ that's compatible with both $E_0$ and $E_1$. In particular, we have $\ord_{E_i}(D)=S(\Vkb;E_i)$ for both $i=0,1$. By our choice of $E_i$ and the definition of $\delta(\Vkb/\Ykb)$, we see that (in a neighbourhood of $X_\eta$)
\[
\delta=\frac{A_{\Xkb,\Deltakb}(E_i)}{S(\Vkb;E_i)}=\frac{A_{\Xkb,\Deltakb}(E_i)}{\ord_{E_i}(D)}\ge \lct(\Xkb,\Deltakb;D)\ge \delta
\]
for $i=0,1$. It follows that $\lct(\Xkb,\Deltakb;D)=\delta$ along $X_\eta$ and is computed by both divisors $E_i$. Clearly $g(Z)=C_{\Xkb}(E_1)$. Suppose that $C_{\Xkb}(E_0)\neq C_{\Xkb}(E_1)$. After possibly shrinking to an open neighbourhood of $\eta\in U$, we may assume that $U$ is affine (we don't need the $G$-action any more), $(\Xkb,\Deltakb+\delta D)$ is lc and $C_{\Xkb}(E_i)$ are minimal lc centers of $(\Xkb,\Deltakb+\delta D)$ (otherwise the dimension of minimal lc center can be made smaller), hence $C_{\Xkb}(E_0)$ is disjoint from $C_{\Xkb}(E_1)$ by \cite{Kol-mmp}*{Corollary 4.41}. Let $H\sim_\bQ -(K_{\Xkb}+\Deltakb+\delta c_1(\Vkb))$ be general among effective divisors that contains $C_{\Xkb}(E_0)\cup C_{\Xkb}(E_1)$ in its support (it is enough to ensure that $\Supp(H)$ doesn't contain other lc centers of $(\Xkb,\Deltakb+\delta D)$). Then for $0<\epsilon\ll \gamma\ll 1$, we have 
\[
\Nklt(\Xkb,\Deltakb+(1-\epsilon)\delta D+\gamma H)=C_{\Xkb}(E_0)\cup C_{\Xkb}(E_1)
\]
and
\[
-(K_{\Xkb}+\Deltakb+(1-\epsilon)\delta D+\gamma H)\sim_\bQ -(1-\gamma)(K_{\Xkb}+\Deltakb+\delta D)+\epsilon\delta D
\]
is ample. This contradicts Koll\'ar-Shokurov's connectedness theorem \cite{K+-flip}*{Theorem 17.4}. Hence we must have $Z=C_{\Xkb}(E_0)=C_{\Xkb}(E_1)=g(Z)$. As $g\in \Gal(\kbar/k)\times G(\kbar)$ is arbitrary, we see that $Z$ is defined over $k$ and $G$-invariant.
\end{proof}

Given this $G$-invariant center, we would like to find a $G$-invariant geometrically irreducible divisor $E$ that computes the `$G$-equivariant' $\delta$-invariant. A priori, there are two ways to define such equivariant $\delta$-invariants: one using $G$-invariant geometrically irreducible divisors over $X$ and the other using $G$-invariant filtrations of the linear series. The second version is more suitable for induction but it does not see the $G$-invariant divisor over $X$ directly. Therefore, the second step in our proof is to compare these two definitions and show that they are actually equivalent. To state the result, we need some more definitions.

Let $V=\sum a_i V_i$ be a boundary on $X$. 
For any filtration $\cF$ of $V$, we define $\fa(\cF)$ to be the $\bQ$-ideal given by 
\[
\fa(\cF):=\prod_{i,\lambda} \fb(\cF^\lambda V_i)^{\frac{a_i}{\dim V_i}\cdot \dim \Gr^\lambda_\cF V_i}
\]
where $\fb(\cF^\lambda V_i)$ denotes the base ideal of the linear series $\cF^\lambda V_i$. 

\begin{defn}
Let $G$ be an algebraic group, let $(X,\Delta)$ be a quasi-projective pair with a $G$-action and let $V$ be a $G$-invariant boundary on $X$. Let $Z\subseteq X$ be a $G$-invariant geometrically irreducible subvariety such that $(X,\Delta)$ is klt along the generic point of $Z$. We set 
\begin{equation} \label{eq:G-delta defn via filtration}
    \tdelta_{Z,G}(V)=\tdelta_{Z,G}(X,\Delta;V):=\inf_{\cF} \lct_Z(X,\Delta;\fa(\cF))
\end{equation}
where the infimum runs over all $G$-invariant filtrations $\cF$ on $V$. We also define 
\begin{equation} \label{eq:G-delta defn via divisor}
    \delta_{Z,G}(V)=\delta_{Z,G}(X,\Delta;V):=\inf_E \frac{A_{X,\Delta}(E)}{S(V;E)}
\end{equation}
where the infimum runs over all $G$-invariant geometrically irreducible divisors $E$ over $X$ whose center contains $Z$.
\end{defn}

\begin{lem} \label{lem:G-filt vs G-div for delta}
In the above notation, we have $\delta_{Z,G}(V)=\tdelta_{Z,G}(V)$. Moreover, the infimum in \eqref{eq:G-delta defn via divisor} is achieved by some $G$-invariant divisor over $X$ that's of plt type at the generic point of $Z$.
\end{lem}

Recall that given a klt pair $(X,\Delta)$, a divisor $E$ over $X$ is said to be of plt type at some $x\in C_X(E)$ if there exists an open neighbourhood $U\subseteq X$ of $x$ and a proper birational morphism $\phi\colon Y\to U$ such that $E$ is a geometrically irreducible divisor on $Y$, $(Y,\Delta_Y+E)$ is plt (where $\Delta_Y$ is the strict transform of $\Delta$) and $-E$ is $\phi$-ample. The map $\phi\colon Y\to U$ is called the associated plt blowup.

\begin{proof}
Let $E$ be a $G$-invariant geometrically irreducible divisor over $X$ whose center contains $Z$. Such divisor always exists by the following Lemma \ref{lem:G-kollar-comp} (e.g. consider divisors that computes $\lct_Z(X,\Delta;\cI_Z)$ where $\cI_Z$ denotes the ideal sheaf of $Z$). Then it induces a $G$-invariant filtration $\cF=\cF_{\ord_E}$ on $V$. It is not hard to see from the definition that  $S(V;E)=\ord_E(\fa(\cF))$, thus
\[
\tdelta_{Z,G}(V)\le \lct_Z(X,\Delta;\fa(\cF))\le  \frac{A_{X,\Delta}(E)}{\ord_E(\fa(\cF))}=\frac{A_{X,\Delta}(E)}{S(V;E)}.
\]
Taking the infimum over all such $E$ we obtain 
$\delta_{Z,G}(V)\ge \tdelta_{Z,G}(V)$. It remains to prove the reverse inequality. To this end, let $\cF$ be a $G$-invariant filtrations on $V$ that achieves the infimum in \eqref{eq:G-delta defn via filtration}; this is possible since such filtrations are parametrized by a closed subset of a flag variety and lct is constructible in family. If $Z$ is not contained in the co-support of $\fa(\cF)$ then $\tdelta_{Z,G}(V)=\lct(X,\Delta;\fa(\cF))=+\infty$ and there is nothing to prove as clearly $\tdelta_{Z,G}(V)\ge \delta_{Z,G}(V)$. Thus we may assume that $Z\subseteq \Cosupp(\fa(\cF))$.
By the following Lemma \ref{lem:G-kollar-comp}, $\lct_Z(X,\Delta;\fa(\cF))$ is computed by some $G$-invariant divisor $E$ over $X$ that's of plt type at the generic point of $Z$. 
Let $m\gg 0$, let $D_1,\cdots,D_m$ be general basis type $\bQ$-divisors of $V$ that are compatible with $\cF$, and let $D=\frac{1}{m}(D_1+\cdots+D_m)$. Then $\lct(X,\Delta;\fa(\cF))=\lct_Z(X,\Delta;D)$ and we have $S(V;E)\ge \ord_E(D)$ since $S(V;E)\ge \ord_E(D_i)$ for each $i=1,\cdots,m$. It follows that
\[
\tdelta_{Z,G}(V)=\lct(X,\Delta;\fa(\cF))=\lct_Z(X,\Delta;D)=\frac{A_{X,\Delta}(E)}{\ord_E (D)}\ge \frac{A_{X,\Delta}(E)}{S(V;E)} \ge \delta_{Z,G}(V).
\]
Combining with the inequality in the opposite direction we finish the proof.
\end{proof}

The following result is used in the above proof.

\begin{lem} \label{lem:G-kollar-comp}
Let $G$ be an algebraic group and let $(X,\Delta)$ be a klt pair endowed with a $G$-action. Let $\fa\subseteq \cO_X$ be a $G$-invariant $\bQ$-ideal and let $Z\subseteq \Cosupp(\fa)$ be a $G$-invariant geometrically irreducible subvariety. Then $\lct_Z(X,\Delta;\fa)$ is computed by some $G$-invariant divisor $E$ over $X$ that's of plt type at the generic point of $Z$.
\end{lem}

\begin{proof}
This is quite standard and should be well known to experts (c.f. \cite{HX-rc-degeneration}*{Proof of Theorem 1.3}). We provide a proof here for reader's convenience. 
For ease of notation we assume that $\Delta=0$; the proof of the general case is the same. Up to shrinking $X$ and replacing $\fa$ by $\fa^\lambda$ (for some $\lambda\in\bQ_+$) we may also assume that $\lct_Z(X;\fa)=\lct(X;\fa)=1$ and all lc centers of $(X,\fa)$ contain $Z$. Let $W$ be the (unique) minimal lc center of $(X,\fa)$ containing $Z$ (see \cite{Kol-mmp}*{Corollary 4.41}), which is necessarily defined over the base field $k$ and $G$-invariant. Let $\cI$ be the ideal sheaf of $W$ and let $\pi\colon Y\to X$ be a common $G$-equivariant log resolution of $(X,\fa)$ and $(X,\cI)$ such that
\begin{enumerate}
    \item the exceptional locus $\Ex(\pi)\subseteq Y$ supports a $G$-invariant $\pi$-ample divisor $A$ (note that $-A$ is effective by the negativity lemma \cite{KM98}*{Lemma 3.39}),
    \item every irreducible component of $\Ex(\pi)\cup \Supp(\pi^*\fa)$ is smooth (i.e. it is a log resolution over $k$) and disjoint from its $G$-translates (these can be achieved by further blowing up strata of $\Ex(\pi)\cup \Supp(\pi^*\fa)$).
\end{enumerate}
Let $0<\epsilon\ll 1$ and write
\begin{equation} \label{eq:K+(1-epsilon)D}
    K_Y+\sum_{i=1}^p a_i(\epsilon) E_i+ \sum_{j=1}^q b_j(\epsilon) F_j +\sum_{k=1}^r c_k(\epsilon) G_k=\pi^*(K_X+\fa^{1-\epsilon})
\end{equation}
where $E_1,\cdots,E_p$ are prime divisors on $Y$ with center $W$ such that $A_{X}(E_i)=\ord_{E_i}(\fa)$, $F_1,\cdots,F_q$ are divisors with center $W$ such that $A_{X}(E_i)>\ord_{E_i}(\fa)$ and $G_1,\cdots,G_r$ are divisors on $Y$ whose centers are different from $W$ (but may contain $W$). We then have $\pi^*\cI=\cO_Y(-\sum_i a'_i E_i-\sum_j b'_j F_j)$ for some $a'_i,b'_j>0$. Note that $\lim_{\epsilon\to 0} a_i(\epsilon)=1>\lim_{\epsilon\to 0} b_j(\epsilon)$ ($\forall i,j$) by construction, thus we have 
\[
\lct(X,\fa^{1-\epsilon};\cI)=\min_{1\le i\le p}\frac{1-a_i(\epsilon)}{a'_i}<\frac{1-b_i(\epsilon)}{b'_j} \;(\forall j)
\]
when $\epsilon$ is sufficiently small. It follows that for such $\epsilon$, if we set $\lambda=\lct(X,\fa^{1-\epsilon};\cI)$, then $(X,\fa^{1-\epsilon}\cdot\cI^\lambda)$ is lc with a unique lc center $W$ and every lc place of $(X,\fa^{1-\epsilon}\cdot\cI^\lambda)$ is also a lc place of $(X,\fa)$. Replacing $\fa$ by $\fa^{1-\epsilon}\cdot\cI^\lambda$, we may assume that $(X,\fa)$ is lc with a unique lc center $W$. In particular, we have $\lim_{\epsilon\to 0} c_k(\epsilon)<1$ for all $k$ in \eqref{eq:K+(1-epsilon)D}. Let $m$ be a sufficiently large and divisible integer and let $\fb=\pi_*\cO_Y(mA)$.
Then $\ord_E(\fb)=\ord_E(-mA)$ for any divisor $E$ over $X$ and as before we see that $\lct(X,\fa^{1-\epsilon};\fb)$ is computed by some divisor $E_i$ when $0<\epsilon\ll 1$. By perturbing the coefficients of $A$ in a $G$-equivariant manner, we can arrange that $\lct(X,\fa^{1-\epsilon};\fb)$ is computed by a unique $G$-orbit $G\cdot E_i$. Replacing $\fa$ by $\fa^{1-\epsilon}\cdot \fb^{\lct(X,\fa^{1-\epsilon};\fb)}$, we may further assume that $(X,\fa)$ is lc whose only lc places are the irreducible components of $G\cdot E_i$. By construction, the irreducible components of $G\cdot E_i$ over $\bar{k}$ are disjoint from each other, thus by the Koll\'ar-Shokurov's connectedness theorem \cite{K+-flip}*{Theorem 17.4}, $G\cdot E_i$ is geometrically irreducible. In other words, $(X,\fa)$ has a unique lc place $E=G\cdot E_i$. By \cite{BCHM}*{Corollary 1.4.3}, there exists a proper birational morphism $\phi\colon \tX\to X$ that extracts $E$ as the unique exceptional prime divisor and $-E$ is $\phi$-ample (such $\tX$ is uniquely determined by $E$, hence is defined over the same base field $k$). Since $E$ is the unique lc place of $(X,\fa)$, we have $\phi^*(K_X+\fa)\ge K_{\tX}+E$ and $E$ is also the unique lc place of $(\tX,E)$. Hence $(\tX,E)$ is plt. In other words, $E$ is of plt type over $X$. Since $E$ is $G$-invariant by construction, we are done.
\end{proof}

We now come to the key inductive step in our proof of Theorem \ref{thm:G-delta=delta for boundary}. Using Lemma \ref{lem:G-filt vs G-div for delta}, we may choose a $G$-invariant divisor $E$ of plt type that computes the $G$-equivariant $\delta$-invariant. Using inversion of adjunction and techniques from \cite{AZ-K-adjunction}, we next compare the $G$-equivariant $\delta$-invariant of $V$ with its `filtered restriction' to $E$ (to be defined in the proof below). Roughly speaking, the consequence is that the filtered restriction of $V$ to $E$ is $G$-equivariantly K-semistable; since $E$ has smaller dimension, we can use our inductive hypothesis to conclude that the filtered restriction is indeed geometrically K-semistable. Using inversion of adjunction and \cite{AZ-K-adjunction} again but without the equivariant information this time, we conclude that the geometric $\delta$-invariant of the origin linear series $V$ is also computed by $E$. These observations lead to the following statement.

\begin{lem} \label{lem:inductive step}
Assume Theorem \ref{thm:G-delta=delta for boundary} for pairs of dimension $n-1$. Let $G$ be an algebraic group, let $(X,\Delta)$ be a klt pair of dimension $n$ with a $G$-action, let $V$ be a $G$-invariant boundary on $X$ and let $Z\subseteq X$ be a $G$-invariant geometrically irreducible subvariety. Assume that $\delta_{\Zkb}(\Vkb)<+\infty$. Then it is computed by some $G$-invariant divisor $E$ over $X$ that's of plt type at the generic point of $Z$, i.e.
\[
\delta_{\Zkb}(\Vkb)=\frac{A_{X,\Delta}(E)}{S(V;E)}.
\]
\end{lem}

\begin{proof}
Since $\delta_{\Zkb}(\Vkb)<+\infty$, there exists some basis type $\bQ$-divisor of $\Vkb$ whose support contains $Z$. Thus if we let $\cF_Z$ be the $G$-invariant filtration of $V$ induced by $\mult_Z$, then $Z\subseteq \Cosupp(\fa(\cF_Z))$ and hence $\tdelta_{Z,G}(V)<+\infty$. We may assume that $Z$ is not a divisor in $X$, otherwise we can clearly take $E=Z$. By Lemma \ref{lem:G-filt vs G-div for delta}, we have
\[
\delta:=\delta_{Z,G}(V)=\tdelta_{Z,G}(V)=\frac{A_{X,\Delta}(E)}{S(V;E)}
\]
for some $G$-invariant divisor $E$ over $X$ with $Z\subseteq C_X(E)$ that's of plt type at the generic point of $Z$. We will show that $\delta_{\Zkb}(\Vkb)$ is also computed by this divisor $E$. Since clearly $\delta_{\Zkb}(\Vkb)\le \delta$, it suffices to prove the reverse inequality.

To this end, let $\cF_E$ be the ($G$-invariant) filtration on $V$ induced by $E$. Replacing $X$ by a $G$-invariant open subset that intersects $Z$, we may assume that $E$ is of plt type over $X$. Let $\pi\colon Y\to X$ be the associated plt blowup. Note that $E$ is exceptional over $X$. We define a boundary $W$ on $E\subseteq Y$ (the `filtered restriction' of $V$ to $E$) as follows: if $V_i\subseteq H^0(X,D_i)$ is a linear series on $X$, its filtered restriction is set to be 
\[
W_i:=\sum_{\lambda\in\bQ} \frac{\dim \Gr^\lambda_{\cF_E} V_i}{\dim V_i}\cdot V_i(-\lambda E)|_E
\]
where $V_i(-\lambda E)$ is the linear series in $H^0(Y,\pi^*D-\lambda E)$ given by $\Gr^\lambda_{\cF_E} V_i$; we then extend the definition to boundaries by taking linear combination. The coefficients in the above definition is chosen such that if $W$ is the filtered restriction of $V$ to $E$ and $D$ is a basis type $\bQ$-divisor of $V$ that's compatible with $E$, then we have $\pi^*D=S(V;E)\cdot E+\Gamma$ for some divisor $\Gamma$ whose support doesn't contain $E$ and $\Gamma|_E$ is a basis type $\bQ$-divisor of $W$; conversely, every basis type $\bQ$-divisor of $W$ can be obtained in this way. In particular, $c_1(W)\sim_\bQ (\pi^*c_1(V)-S(V;E)\cdot E)|_E$ is $\pi$-ample. 

Next, let $F$ be a $G$-invariant geometrically irreducible divisor over $E$ whose center dominates $Z$, let $m\gg 0$, and let $D_0^{(i)}$ ($i=1,\cdots,m$) be general basis type $\bQ$-divisors of $W$ that are compatible with $F$. Let $D_0=\frac{1}{m}(D_0^{(1)}+\cdots+D_0^{(m)})$. By construction, the filtration on $W$ induced by $F$ lifts to a ($G$-invariant) refinement $\cF$ of the filtration $\cF_E$ on $V$, and $D_0$ lifts to a divisor $D$ on $X$ that's a convex linear combination of general basis type $\bQ$-divisors of $V$ that are compatible with $\cF$ (hence are also compatible with $\cF_E$). We have
\[
\delta=\tdelta_{Z,G}(V)=\frac{A_{X,\Delta}(E)}{S(V;E)}=\frac{A_{X,\Delta}(E)}{\ord_E(D)}\ge \lct_Z(X,\Delta;D)=\lct_Z(X,\Delta;\fa(\cF))\ge \tdelta_{Z,G}(V),
\]
thus equality holds everywhere and $E$ computes $\delta=\lct_Z(X,\Delta;D)$. Let $\Delta_E={\rm Diff}_E(\Delta_Y)$ be the different; as $E$ is of plt type over $X$, $(E,\Delta_E)$ is klt and $-(K_E+\Delta_E)$ is $\pi$-ample. Recall that $\pi^*D=S(V;E)\cdot E + \Gamma$ where $\Gamma|_E=D_0$. We thus have $\pi^*(K_X+\Delta+\delta D)=K_Y+\Delta_Y+E+\delta\Gamma$; hence $K_E+\Delta_E+\delta c_1(W)\sim_\bQ (K_Y+\Delta_Y+E+\delta\Gamma)|_E \sim_{\pi,\bQ} 0$ and by inversion of adjunction we deduce that $(E,\Delta_E+\delta D_0)$ is lc over the generic point of $Z$. It follows that
\begin{equation} \label{eq:G-inv A/S>=outside delta}
    \frac{A_{E,\Delta_E}(F)}{S(W;F)}=\frac{A_{E,\Delta_E}(F)}{\ord_F(D_0)}\ge \delta
\end{equation}
for all $G$-invariant geometrically irreducible divisors over $E$ whose centers dominate $Z$. Suppose that 
\[
\delta_0:=\delta(E_{\kbar},(\Delta_E)_{\kbar};W_{\kbar}/\Zkb) < \delta.
\]
Then $-(K_E+\Delta_E+\delta_0 c_1(W))\sim_{\pi,\bQ} (\delta-\delta_0)c_1(W)$ is $\pi$-ample and by Theorem \ref{thm:G-delta=delta for boundary} (noting that $E$ has dimension $n-1$), there exists some $G$-invariant geometrically irreducible divisor $F$ over $E$ whose center dominates $Z$ such that
\[
\frac{A_{E,\Delta_E}(F)}{S(W;F)}=\delta_0<\delta,
\]
which contradicts \eqref{eq:G-inv A/S>=outside delta}. Therefore, we must have $\delta(E_{\kbar},(\Delta_E)_{\kbar};W_{\kbar}/\Zkb)\ge \delta$. But then by \cite{AZ-K-adjunction}*{Proof of (3.4)}, we get 
\[
\delta_{\Zkb}(\Xkb,\Deltakb;\Vkb)\ge \min\left\{\frac{A_{X,\Delta}(E)}{S(V;E)},\delta(E_{\kbar},(\Delta_E)_{\kbar};W_{\kbar}/\Zkb)\right\}\ge \delta
\]
as desired.
\end{proof}

We are now ready to prove

\begin{proof}[Proof of Theorem \ref{thm:G-delta=delta for boundary}]
We prove by induction on $n=\dim X$. The base case $n=0$ is empty. Suppose the statement has been proved in dimension $n-1$. By Lemma \ref{lem:G-inv center}, there exists some $G$-invariant geometrically irreducible subvariety $Z\subseteq X$ dominating $Y$ such that $\delta(\Vkb/\Ykb)=\delta_{\Zkb}(\Vkb)<\infty$. By induction hypothesis and Lemma \ref{lem:inductive step}, there exists some $G$-invariant geometrically irreducible divisor $E$ over $X$ whose center contains $Z$ such that $\delta_{\Zkb}(\Vkb)=\frac{A_{X,\Delta}(E)}{S(V;E)}$. Thus 
\[
\delta(\Vkb/\Ykb)=\frac{A_{X,\Delta}(E)}{S(V;E)},
\]
proving the statement in dimension $n$.
\end{proof}

\begin{proof}[Proof of Theorem \ref{thm:G-delta=delta}]
Let $\epsilon>0$. By \cite{BJ-delta}*{Corollary 3.6}, we have $S_m(v)\le (1+\epsilon)S(v)$ for all $m\gg 0$ and all $v\in\Val^*_X$. Let $m\gg 0$ be also sufficiently divisible such that $\delta_m:=\delta_m(\Xkb,\Deltakb)<1$ and let $V_m=\frac{1}{m}|-m(K_X+\Delta)|$ be the complete $\bQ$-linear series. Then we have $\delta((V_m)_{\kbar})=\delta_m$ by definition and hence $-(K_X+\Delta+\delta((V_m)_{\kbar})c_1(V_m))\sim_\bQ -(1-\delta_m)(K_X+\Delta)$ is ample. By Theorem \ref{thm:G-delta=delta for boundary} (applied to the pair $(X,\Delta)$ over $U={\rm point}$ with boundary $V_m$), we see that there exists some $G$-invariant geometrically irreducible divisor $E$ over $X$ such that
\[
\frac{A_{X,\Delta}(E)}{S_m(E)}=\frac{A_{X,\Delta}(E)}{S(V_m;E)}=\delta_m.
\]
Let $D$ be an $m$-basis type $\bQ$-divisor of $(X,\Delta)$ that's compatible with $E$. Then $E$ computes $\lct(X,\Delta;D)=\delta_m<1$ by the definition of $\delta_m$ and the above equality. It follows that $E$ is an lc place of the complement $\delta_m D+(1-\delta_m)H$ where $H\sim_\bQ -(K_X+\Delta)$ is effective and general. We also have
\[
\delta(\Xkb,\Deltakb)\le \frac{A_{X,\Delta}(E)}{S(E)}\le (1+\epsilon)\frac{A_{X,\Delta}(E)}{S_m(E)}=(1+\epsilon)\delta_m
\]
for $m\gg 0$. As $\epsilon$ is arbitrary and $\lim_{m\to \infty} \delta_m = \delta(\Xkb,\Deltakb)$, the equality \eqref{eq:delta=inf over G-inv div} follows.
\end{proof}

\begin{rem}
Using the argument of \cite{BLX-openness}, one can further show that $\delta(\Xkb,\Deltakb)$ is computed by some $\Aut(X,\Delta)$-invariant quasi-monomial valuation that's an lc place of a bounded complement defined over $k$. We leave the details to the reader.
\end{rem}

As in \cite{LZ-equivariant}, Theorem \ref{thm:G-delta=delta} implies the equivalence of equivariant K-semistability (resp. K-polystability) with geometric K-semistability (resp. K-polystability), as well as a generalization of Tian's criterion.

\begin{cor} \label{cor:equivariant}
Let $G$ be an algebraic group and let $(X,\Delta)$ be a log Fano pair with an action of $G$.
\begin{enumerate}
    \item If $(X,\Delta)$ is $G$-equivariantly K-semistable, then $(\Xkb,\Deltakb)$ is K-semistable.
    \item If $G$ is reductive and $(X,\Delta)$ is $G$-equivariantly K-polystable, then $(\Xkb,\Deltakb)$ is K-polystable.
\end{enumerate}
\end{cor}

The following example shows that the reductivity of $G$ is necessary.

\begin{expl} \label{ex:non-reductive}
Let $k=\bC$ and let $X^{\ra}$ be the unique Fano threefold of degree $22$ whose identity component $\Aut(X^{\ra})_0$ of the automorphism group is isomorphic to the additive group $\bC^+$ (\cite{Pro-V22}). As discussed in \cite{CS-KE-V22}*{Example 1.4}, $X^{\ra}$ is K-semistable but not K-polystable and its K-polystable degeneration is the Mukai-Umemura threefold $X^{\rm MU}$. Moreover, $X^{\rm MU}$ is the only nontrivial K-semistable special degeneration of $X^{\ra}$. To see this, let $Y$ be a nontrivial K-semistable special degeneration of $X^{\ra}$. Then it has a faithful $\bC^*$-action. If $Y$ is not isomorphic to $X^{\rm MU}$, then by \cite{LWX-Kps-degeneration}*{Theorems 1.4 and 3.2}, $Y$ admits a nontrivial $\bC^*$-equivariant special degeneration to $X^{\rm MU}$, giving rise to a faithful $(\bC^*)^2$-action on $X^{\rm MU}$, which is impossible as $\Aut(X^{\rm MU})\cong {\rm PGL}(2,\bC)$. Thus $Y\cong X^{\rm MU}$. However, none of the special degenerations of $X^{\ra}$ to $X^{\rm MU}$ is $\bC^+$-equivariant: otherwise we get a faithful $\bC^+\times\bC^*$-action on $X^{\rm MU}$, which is again impossible. It follows that $X^{\ra}$ is $\bC^+$-equivariantly K-polystable but not K-polystable by Theorem \ref{thm:test G-Kps by stc}.
\end{expl}

\begin{proof}[Proof of Corollary \ref{cor:equivariant}]
We first prove the K-semistable part. Suppose that $(\Xkb,\Deltakb)$ is not K-semistable, i.e. $\delta(\Xkb,\Deltakb)<1$. Then by Theorem \ref{thm:G-delta=delta} and Remark \ref{rem:dreamy}, there exists some $G$-invariant geometrically irreducible dreamy divisor $E$ over $X$ such that $A_{X,\Delta}(E)<S(E)$, which implies that $(X,\Delta)$ is not $G$-equivariantly K-semistable by Proposition \ref{prop:G-valuative criterion}, a contradiction. Thus $(\Xkb,\Deltakb)$ is K-semistable.

Assume next that $G$ is reductive and $(X,\Delta)$ is $G$-equivariantly K-polystable. Then from the previous part we know that $(\Xkb,\Deltakb)$ is K-semistable. Let $(X_0,\Delta_0)$ be its unique K-polystable degeneration \cite{LWX-Kps-degeneration}*{Theorem 1.3} (a priori it is only defined over $\kbar$). Let 
\[
W:=\overline{{\rm PGL}_{N+1}(\kbar)\cdot [(X,\Delta)]}\subseteq \left({\rm Hilb}(\bP^N)\times {\rm Chow}(\bP^N)\right)\cap \cW^{\rm Kss}
\]
and $W_0:={\rm PGL}_{N+1}(\kbar)\cdot [(X_0,\Delta_0)]\subseteq W$ be the corresponding locus (with reduced scheme structure) in the moduli of K-semistable log Fano pairs (c.f. \cite{XZ-CM-positive}*{Proof of Theorem 2.21}). Then $W_0$ is closed in $W$: otherwise there exists some K-semistable log Fano pair $(X_1,\Delta_1)$ such that $W_1\subsetneq \overline{W_0}\cap W$ where $W_1:={\rm PGL}_{N+1}(\kbar)\cdot [(X_1,\Delta_1)]$; in particular, $\dim W_1<\dim W_0$ and $(X_0,\Delta_0)$ has an isotrivial degeneration to $(X_1,\Delta_1)$; since $(X_0,\Delta_0)$ is K-polystable, we deduce that $(X_1,\Delta_1)$ also specially degenerates to $(X_0,\Delta_0)$ by \cite{BX-uniqueness}*{Theorem 1.1(1) and Remark 1.2(2)}, thus $W_0\subseteq \overline{W_1}$ and $\dim W_0\le \dim W_1$, a contradiction. We also have $W_0$ is defined over $k$ since all Galois conjugates of $(X_0,\Delta_0)$ are also K-polystable degenerations of $(\Xkb,\Deltakb)$, hence are isomorphic to $(X_0,\Delta_0)$ by the uniqueness of K-polystable degeneration. Since $x:=[(X,\Delta)]\in W$ is a $k$-rational point that's fixed by the reductive group $G$, by \cite{Kempf}*{Corollary 4.5} we see that there exists a 1-parameter subgroup $\rho\colon \bG_m\to Z_G({\rm PGL}_{N+1})$ (defined over $k$) such that $\lim_{t\to 0} \rho(t)\cdot x\in W_0$. In other words, there exists a $G$-equivariant special test configuration of $(X,\Delta)$ defined over $k$ whose central fiber is isomorphic to $(X_0,\Delta_0)$ over the algebraic closure $\kbar$; in particular, the central fiber is geometrically K-semistable. Since $(X,\Delta)$ is $G$-equivariantly K-polystable, we have $(\Xkb,\Deltakb)\cong (X_0,\Delta_0)$ by Theorem \ref{thm:test G-Kps by stc} and hence $(\Xkb,\Deltakb)$ is K-polystable.
\end{proof}

\begin{cor}[\cite{LZ-equivariant}*{Theorem 1.2}] \label{cor:finite Galois}
Let $\pi\colon (X,D)\to (Y,B)$ be a finite surjective Galois morphism between log Fano pairs such that $K_X+D=\pi^*(K_Y+B)$. Then
\begin{enumerate}
    \item $(X,D)$ is K-semistable $($resp. K-polystable$)$ if and only if $(Y,B)$ is K-semistable $($resp. K-polystable$)$.
    \item If one of $(X,D)$ or $(Y,B)$ is K-unstable, then $\delta(X,D)=\delta(Y,B)$.
\end{enumerate}
\end{cor}

\begin{proof}
Let $G=\Aut(f)$ be the Galois group. By Theorem \ref{thm:G-delta=delta} and Corollary \ref{cor:equivariant}, we may assume that the base field is algebraically closed. By \cite{Fuj-plt-blowup}*{Corollary 1.7} and the computations in \cite{Fuj-plt-blowup}*{Section 4.1}, we know that $\delta(X,D)\le \delta(Y,B)$ and $(Y,B)$ is K-semistable if $(X,D)$ is. For every $G$-invariant prime divisor $E$ over $X$, there exists a $G$-equivariant birational morphism $X'\to X$ and a prime divisor $F$ on $Y'=X'/G$ such that $E$ is a divisor on $X'$ and $\Supp(\pi^*F)=E$ (here we denote the induced map $X'\to Y'$ also by $\pi$). A direct computation as in \cite{Fuj-plt-blowup}*{Section 4.1} shows that 
\[
\frac{A_{X,D}(E)}{S_{X,D}(E)}=\frac{A_{Y,B}(F)}{S_{Y,B}(F)},
\]
thus by Theorem \ref{thm:G-delta=delta} we obtain $\delta(X,D)\ge \delta(Y,B)$ as long as $(X,D)$ is K-unstable. Hence $\delta(X,D)=\delta(Y,B)$ in this case. This proves (2) and also implies the K-semistable part in (1). The K-polystable part in (1) then follows from the same argument as in \cite{LZ-equivariant}*{Theorem 1.2(2)}.
\end{proof}

\begin{cor} \label{cor:valuative criterion}
Let $G$ be an algebraic group and let $(X,\Delta)$ be a log Fano pair with a $G$-action. Assume that $A_{X,\Delta}(E)\ge S(E)$ $($resp. $G$ is reductive and $A_{X,\Delta}(E) > S(E))$ for all $G$-invariant geometrically irreducible divisors $E$ over $X$. Then $(\Xkb,\Deltakb)$ is K-semistable $($resp. K-polystable$)$.
\end{cor}

\begin{proof}
We only prove the K-polystable part since the K-semistable part follows directly from Theorem \ref{thm:G-delta=delta}. Assume that $G$ is reductive and $A_{X,\Delta}(E) > S(E)$ for all $G$-invariant geometrically irreducible divisors $E$ over $X$. Then $(\Xkb,\Deltakb)$ is K-semistable. If $(\Xkb,\Deltakb)$ is not K-polystable, then as in the proof of Corollary \ref{cor:equivariant} we see that $(X,\Delta)$ has a non-trivial $G$-equivariant special test configuration with K-semistable central fiber. This is induced by a $G$-invariant geometrically irreducible divisor $E$ over $X$ with $A_{X,\Delta}(E)=S(E)$ (see \cite{BHJ-DH-measure} or \cite{Fuj-valuative-criterion}*{Theorem 5.1}), a contradiction to our assumption. Thus $(\Xkb,\Deltakb)$ is K-polystable as desired.
\end{proof}

\begin{cor} \label{cor:Tian's criterion}
Let $G$ be an algebraic group and let $(X,\Delta)$ be a log Fano pair of dimension $n$ with a $G$-action.
\begin{enumerate}
    \item If $\alpha_G(X,\Delta)\ge \frac{n}{n+1}$, then $(\Xkb,\Deltakb)$ is K-semistable.
    \item If $G$ is reductive and $\alpha_G(X,\Delta) > \frac{n}{n+1}$, then $(\Xkb,\Deltakb)$ is K-polystable.
\end{enumerate}
\end{cor}

\begin{proof}
Suppose that $\alpha_G(X,\Delta)\ge \frac{n}{n+1}$. Then we have $A_{X,\Delta}(E)\ge \frac{n}{n+1}T(E)$ for any $G$-invariant geometrically irreducible divisor $E$ over $X$. By \cite{BJ-delta}*{Proposition 3.11}, we also have $S(E)\le \frac{n}{n+1}T(E)$ for any such divisor $E$, thus $A_{X,\Delta}(E)\ge S(E)$ for any $G$-invariant geometrically irreducible divisor $E$ over $X$. By Corollary \ref{cor:valuative criterion}, this implies that $(\Xkb,\Deltakb)$ is K-semistable. The proof of the K-polystable part is similar.
\end{proof}

Using equivariant K-stability and Corollary \ref{cor:equivariant}, we can also give algebraic proofs of the K-stability of some explicit Fano varieties. First we provide a short proof of the K-(poly)stability of del Pezzo surfaces (see e.g. \cites{Tian-dP,Che-lct-dP,PW-dP} for other proofs).

\begin{cor}
Let $X$ be a smooth complex del Pezzo surface of degree $d$. Then $X$ is K-polystable if and only if $X$ is not the blowup of one or two points on $\bP^2$ and it is K-stable if and only if $d\le 5$.
\end{cor}

\begin{proof}
It is well known that the blow up of one or two points on $\bP^2$ is not K-polystable since the automorphism group is not reductive. Let $X=\bP^2$, $\bP^1\times \bP^1$ or the blow up of $\bP^2$ at three points and let $G=\Aut(X)$. Then it is easy to see that there are no $G$-invariant curves or $G$-fixed point on $X$, thus $X$ is K-polystable by Corollary \ref{cor:valuative criterion}. If $X$ has degree $5$ (resp. $4$), then $G=S_5$ (resp. $G=(\bZ/2\bZ)^4$) acts on $X$, $\Pic(X)^G=\bZ\cdot [-K_X]$ and there are no $G$-fixed points on $X$ (see e.g. \cite{DI-cremona}). It follows that every $G$-invariant prime divisor over $X$ is a $G$-invariant curve $C\sim -rK_X$ on $X$ for some $r\ge 1$; an easy calculation shows that $\beta_X(C)=1-\frac{1}{3r}>0$, thus $X$ is K-polystable by Corollary \ref{cor:valuative criterion}. Since $X$ has finite automorphism group in this case, it is in fact K-stable. Finally if $d\le 3$ then $X$ is K-stable by \cite{AZ-K-adjunction}*{Corollary 4.9}.
\end{proof}

As another example, we show that:

\begin{cor} \label{cor:fermat and two quadric}
The following Fano manifolds are K-stable:
\begin{enumerate}
    \item \cite{Tian-Fermat} Fermat hypersurfaces $(x_0^d+\cdots+x_{n+1}^d=0)\subseteq \bP^{n+1}$ $(3\le d\le n+1)$.
    \item \cite{AGP-two-quadric} Complete intersection of two quadrics $Q_1\cap Q_2\subseteq \bP^{n+2}$.
\end{enumerate}
\end{cor}

Combined with \cites{Xu-quasimonomial,BLX-openness}, this gives an algebraic proof that a general Fano hypersurface of degree at least $3$ is K-stable.

\begin{proof}
First let $X$ be the Fermat hypersurfaces $(x_0^d+\cdots+x_{n+1}^d=0)\subseteq \bP^{n+1}$ $(3\le d\le n+1)$ and consider the morphism $f\colon X\to \bP^n$ given by $[x_0:\cdots:x_{n+1}]\mapsto [x_0^d:\cdots:x_n^d]$. Let $y_0,\cdots,y_n$ be the homogeneous coordinates on the target $\bP^n$, let $H_i=(y_i=0)\subseteq \bP^n$ ($i=0,1,\cdots,n$) and let $H_{n+1}=(y_0+\cdots+y_n=0)$. Then it is clear that $f$ is Galois and $K_X=f^*\left(K_{\bP^n}+(1-\frac{1}{d})(H_0+\cdots+H_{n+1})\right)$. By \cite{Fuj-hyperplane}*{Corollary 1.6}, the pair $(\bP^n,(1-\frac{1}{d})(H_0+\cdots+H_{n+1}))$ is K-polystable, thus by Corollary \ref{cor:finite Galois}, the Fermat hypersurface $X$ is also K-polystable. It is indeed K-stable since $\Aut(X)$ is finite when $d\ge 3$. 

The argument is similar for the complete intersection of two quadrics $X=Q_1\cap Q_2\subseteq \bP^{n+2}$. Up to a change of coordinates, we may assume that $Q_1=(x_0^2+\cdots+x_{n+2}^2=0)$ and $Q_2=(a_0 x_0^2+\cdots+a_{n+2} x_{n+2}^2=0)$ for some mutually distinct coefficients $a_0,\cdots,a_{n+2}$. Consider the Galois morphism $g\colon X\to \bP^n$ given by $[x_0:\cdots:x_{n+2}]\mapsto [x_0^2:\cdots:x_n^2]$. We may identify $\bP^n$ with the linear subspace $(x_0+\cdots+x_{n+2}=a_0 x_0+\cdots+a_{n+2} x_{n+2}=0)\subseteq \bP^{n+2}$ and let $H_i\subseteq \bP^n$ ($i=0,1,\cdots,n+2$) be the hyperplane $(x_i=0)$. Then we have $K_X=g^*\left(K_{\bP^n}+\frac{1}{2}(H_0+\cdots+H_{n+2})\right)$. By \cite{Fuj-hyperplane}*{Corollary 1.6}, the pair $(\bP^n,\frac{1}{2}(H_0+\cdots+H_{n+2}))$ is K-polystable, thus by Corollary \ref{cor:finite Galois}, the complete intersection $X$ is also K-polystable. Since $\Aut(X)$ is again finite, we conclude that it is K-stable.
\end{proof}

\bibliography{ref}

\end{document}